\documentclass[12pt]{amsart}
\usepackage{graphicx}
\usepackage{amsmath,amssymb,array}
\usepackage{pgfpages}
\usepackage{color}
\usepackage{tikz}
\usetikzlibrary{arrows}
\usetikzlibrary{calc}

\newtheorem{theorem}{Theorem}[section]

\newtheorem{lemma}[theorem]{Lemma}
\newtheorem{proposition}[theorem]{Proposition}
\newtheorem{corollary}[theorem]{Corollary}

\theoremstyle{remark}

\theoremstyle{definition}
\newtheorem{definition}[theorem]{Definition}

\numberwithin{equation}{section}

\newcommand{\et}{\quad\mbox{and}\quad}

\newcommand{\bP}{\mathbb{P}}
\newcommand{\bQ}{\mathbb{Q}}
\newcommand{\bR}{\mathbb{R}}
\newcommand{\bZ}{\mathbb{Z}}
\newcommand{\cC}{{\mathcal{C}}}

\newcommand{\dist}{\mathrm{dist}}

\newcommand{\proj}{\mathrm{proj}}

\newcommand{\ue}{\mathbf{e}}

\newcommand{\uL}{\mathbf{L}}
\newcommand{\up}{\mathbf{p}}
\newcommand{\uP}{\mathbf{P}}
\newcommand{\uu}{\mathbf{u}}
\newcommand{\uv}{\mathbf{v}}

\newcommand{\ux}{\mathbf{x}}
\newcommand{\uX}{\mathbf{X}}
\newcommand{\uy}{\mathbf{y}}

\newcommand{\vol}{\mathrm{vol}}

\begin{document}

\baselineskip=16.3pt

\title[Points realizing regular systems]
{Construction of points realizing the regular systems
of Wolfgang Schmidt and Leonard Summerer}
\author{Damien Roy}
\address{
   D\'epartement de Math\'ematiques\\
   Universit\'e d'Ottawa\\
   585 King Edward\\
   Ottawa, Ontario K1N 6N5, Canada}
\email{droy@uottawa.ca}
\dedicatory{On the occasion of Axel Thue's 150th birthday,\\
with special homage to Professor Wolfgang Schmidt on his 80th birthday.}
\thanks{Work partially supported by NSERC}

\begin{abstract}
In a series of recent papers, W.~M.~Schmidt and L.~Summerer developed
a new theory by which they recover all major generic inequalities
relating exponents of Diophantine approximation to a point in $\bR^n$,
and find new ones.  Given a point in $\bR^n$, they first show how most
of its exponents of Diophantine approximation can be computed in terms
of the successive minima of a parametric family of convex bodies
attached to that point.  Then they prove that these successive minima
can in turn be approximated by a certain class of functions which they
call $(n,\gamma)$-systems.  In this way, they bring the whole problem
to the study of these functions.  To complete the theory, one would like
to know if, conversely, given an $(n,\gamma)$-system, there exists a
point in $\bR^n$ whose associated family of convex bodies has
successive minima which approximate that function.  In the present
paper, we show that this is true for a class of functions which
they call regular systems.
\end{abstract}

\maketitle

\section{Introduction}
 \label{intro}

Let $\xi_1,\dots,\xi_{n-1}\in\bR$ for some
integer $n\ge 2$.  A basic problem in Diophantine
approximation is to measure how well the point $(\xi_1,\dots,\xi_{n-1})$
can be approximated by rational points with common denominators below
a given bound, and how small can integer linear combinations of
$1,\xi_1,\dots,\xi_{n-1}$ be, given an upper bound on
the absolute values of their coefficients.  This gives
rise to four classical exponents of approximation which are linked
by the dualities of A.~Y.~Khintchine \cite{Kh1926a, Kh1926b} and
V.~Jarn\'{\i}k \cite{Ja1938}.  In the case $n=3$, M.~Laurent
achieved recently a complete description of the joint spectrum
of these four exponents \cite{La2009}. Such a description is still
lacking in higher dimensions. However, N.~Moshchevitin \cite{Mo2012}
recently found a new relation between these exponents in the case
$n=4$.  Then, a second proof of it together with a proof of a ``dual''
relation was given by W.~M.~Schmidt and L.~Summerer in \cite{SS2013b}
using their theory of parametric geometry of numbers.  To show that
both relations are best possible these authors ask for the existence of
points in $\bR^4$ satisfying certain conditions that we will recall
below.  The purpose of this note is to construct such points.
For the interested reader, it can serve
as an introduction to \cite{R_preprint} where we construct points
satisfying the fully general conditions provided by the theory of
Schmidt and Summerer.

This wonderful theory, called parametric geometry of numbers by their
authors, was developed first in dimension $n=3$ in \cite{SS2009}
and then for general dimension $n\ge 2$ in \cite{SS2013a}.  It provides
a very precise description of the behavior of the successive minima of
certain parametric families of convex bodies of $\bR^n$.  Here, the term
\emph{convex body of $\bR^n$} refers to a compact \hbox{$0$-symmetric}
neighborhood $\cC$ of $0$ in $\bR^n$.  We recall that, for $j=1,\dots,n$,
the $j$-th minimum $\lambda_j(\cC)$ of such a set is the smallest
real number $\lambda$ such that $\lambda\cC$ contains at least
$j$ linearly independent elements of $\bZ^n$.  Clearly these minima form a
monotone increasing sequence $\lambda_1(\cC)\le\cdots\le\lambda_n(\cC)$.
Throughout this paper, we assume that the integer $n$ is at least $2$.

Let $\ux\cdot\uy$ denote the usual scalar product of vectors
$\ux, \uy \in \bR^n$, and let $\|\ux\|=(\ux\cdot\ux)^{1/2}$
denote the corresponding norm of a vector $\ux$.
For our purpose, we work with the families of convex bodies
\[
 \cC_\uu(Q)=\big\{\ux\in\bR^n\,;\,
           \|\ux\|\le Q, \
           |\ux\cdot\uu|\le Q^{-(n-1)}\,
        \big\}
 \quad
 (Q\ge 1).
\]
where $\uu$ is a fixed unit vector of $\bR^n$.  These are essentially
the polar reciprocal bodies to those considered in \cite{SS2013a}
but in view of the close relations linking the successive minima of a
convex body to those of its polar reciprocal body, this makes very
little difference.  Besides its own fundamental intrinsic interest,
a strong motivation for studying the successive minima of $\cC_\uu(Q)$
as functions of $Q$ comes from the fact that, if we choose $\uu$
to be a multiple of $(1,\xi_1,\dots,\xi_{n-1})$, then the four
exponents to which we alluded above can be computed directly from
these functions (see \cite[\S1]{SS2013a}), and the same holds for
the intermediate exponents studied by Y.~Bugeaud and M.~Laurent
in \cite{BL2010} (see also \cite{Ge2012} and \cite{Sc1967}).
In fact, let
\[
 \Delta_n
 :=\{(x_1,\dots,x_n)\in\bR^n\,;\, x_1\le \cdots\le x_n\},
\]
and consider the continuous map $\uL_\uu\colon[0,\infty)\to
\Delta_n$ given by
\[
 L_\uu(q)
 =\big(\log\lambda_1(\cC_\uu(e^q)),
       \dots,
       \log\lambda_n(\cC_\uu(e^q))
  \big)
 \quad
 (q\ge 0).
\]
An approximation of $\uL_\uu$ with bounded
difference suffices by far to compute these exponents.

In \cite[\S2]{SS2013a}, Schmidt and Summerer define, for each
$\gamma\ge 0$ and each $a\ge 0$, the notion of an
$(n,\gamma)$-system on the interval $[a,\infty)$.  This is
a continuous map $\uP\colon[a,\infty)\to\bR^n$ which satisfies a
certain number of conditions which, although relatively easy to state,
are somewhat difficult to analyze.  The largest part of
their paper deals with this issue.  Here, since we
essentially use the polar reciprocal bodies, the relevant
notion for us is a dual one as in \cite[\S7]{SS2013b}.
However, for simplicity, we keep the same terminology.
Then, modulo slight modifications, the authors establish
in \cite[\S2]{SS2013a} the existence of a constant $\gamma>0$
and of an $(n,\gamma)$-system $\uP\colon[0,\infty)\to\bR^n$
such that $\uL_\uu-\uP$ is bounded on $[0,\infty)$.

As shown in \cite[\S3]{SS2013a}, the behavior of an $(n,0)$-system
is much easier to understand.  In particular, such a map takes
values in $\Delta_n$. In \cite{R_preprint}, we show that, for each
$(n,\gamma)$-system $\uP\colon[0,\infty)\to\bR^n$, there exist a real
number $a\ge 0$ and an $(n,0)$-system $\uX\colon[a,\infty)\to\Delta_n$
for which the difference $\uP-\uX$ is bounded on $[a,\infty)$.
In view of the result of Schmidt and Summerer mentioned above,
this means that, for any
unit vector $\uu$ in $\bR^n$, there exists an $(n,0)$-system
$\uX\colon[a,\infty)\to\Delta_n$ for which $\uL_\uu-\uX$ is bounded
on $[a,\infty)$.  In \cite{R_preprint}, we also show that the
converse is true namely that, for each $(n,0)$-system
$\uX\colon[a,\infty)\to\Delta_n$, there exists a unit vector
$\uu$ of $\bR^n$ such that $\uL_\uu-\uX$ is bounded on $[a,\infty)$.
In particular, this proves a conjecture of \cite[\S4]{SS2013a}
to the effect that all generic relations between
exponents of approximation can be derived from the study
of $(n,0)$-systems.

Our goal here is to construct unit vectors associated to a
class of $(n,0)$-systems which is slightly more general than
the \emph{regular systems} of \cite[\S3]{SS2013b}.  To present
this class of \emph{quasi-regular $(n,0)$-systems}, we
follow Schmidt and Summerer in \cite[\S3]{SS2013a} and
define the \emph{combined graph} of
a set of real valued functions defined on an interval $I$ to
be the union of their graphs in $I\times\bR$.  For a
function $\uP=(P_1,\dots,P_n)\colon [a,\infty)\to \Delta_n$,
and a sub-interval $I$ of $[a,\infty)$, we define the
\emph{combined graph of\/ $\uP$ above $I$} to be the combined
graph of its components $P_1,\dots,P_n$ restricted to
$I$.  If $P$ is continuous and if the real numbers
$q\ge a$ at which $P_1(q),\dots,P_n(q)$ are not all distinct
form a discrete subset of $[a,\infty)$, then the map
$\uP$ is uniquely determined by its combined graph over the
full interval $[a,\infty)$.  We also denote by
$\Phi_n\colon\bR^n\to\Delta_n$ the continuous
map which lists the coordinates of a point in monotone
increasing order.

\begin{definition}
 \label{def}
A \emph{quasi-regular $(n,0)$-system} is a continuous function
$\uP\colon[a,\infty)\to\Delta_n$ for which
there exists an unbounded strictly increasing sequence
of positive real numbers $(X_i)_{i\ge 1}$ such that,
upon defining
\[
 q_i=(X_i+\cdots+X_{i+n-1})/n
 \quad
 (i\ge 1),
\]
we have $a=q_1$ and, for each $i\ge 1$,
\begin{equation}
 \label{intro:eq:uP}
 \uP(q)=\Phi_n(X_i+n(q-q_i)-q,\,X_{i+1}-q,\,\dots,\,X_{i+n-1}-q)
 \quad
 (q_i\le q\le q_{i+1}).
\end{equation}
If, for some $\delta>0$, we also have
$X_{i+1}\ge X_i+\delta$ for each $i\ge 1$, then we say that
\emph{$\uP$ has mesh at least $\delta$}.  If there
exists $\rho>1$ such that $X_{i+1}=\rho X_i$ for each $i\ge 1$,
then we say that $\uP$ is \emph{regular}.
\end{definition}

Since $X_i+n(q_{i+1}-q_i)-q_{i+1}=X_{i+n}-q_{i+1}$, the
condition \eqref{intro:eq:uP} implies that
\[
 \uP(q_i)=(X_i-q_i,\dots,X_{i+n-1}-q_i)
 \et
 \uP(q_{i+1})=(X_{i+1}-q_{i+1},\dots,X_{i+n}-q_{i+1}).
\]
Therefore, upon writing $\uP=(P_1,\dots,P_n)$, it
is equivalent to asking that the
combined graph of $\uP$ above $[q_i,q_{i+1}]$ consists of
one line segment of slope $n-1$ joining $(q_i,P_1(q_i))$ to
$(q_{i+1},P_n(q_{i+1}))$, together with $n-1$ distinct line
segments of slope $-1$ joining $(q_i,P_{j+1}(q_i))$ to
$(q_{i+1},P_j(q_{i+1}))$ for $j=1,\dots,n-1$.

The above remark shows in particular that any
choice of $0<X_1<X_2<\cdots$ with $\lim_{i\to\infty}X_i=\infty$
gives rise to a continuous map $\uP\colon[q_1,\infty)\to\Delta_n$
satisfying \eqref{intro:eq:uP} for each $i\ge 1$.  It also implies
that, in turn, such a map $\uP$ uniquely determines the sequence
$(X_i)_{i\ge 1}$ because the local minima of its first component
$P_1$ are the points $(q_i,P_1(q_i))=(q_i,X_i-q_i)$ $(i\ge 1)$.
This is illustrated on Figure 1 below which shows in solid lines
the combined graph of a quasi-regular $(4,0)$-system over an
interval $[q_1,q_5]$.

\begin{figure}[ht]
  \begin{tikzpicture}[xscale=0.25,yscale=0.1]
    \draw [->] (-0.25,0) -- (43,0) node[below] {$q$};
      \pgfmathparse{-10}
        \let\xa\pgfmathresult
      \pgfmathparse{-3}
        \let\xb\pgfmathresult
      \pgfmathparse{3}
        \let\xc\pgfmathresult
      \pgfmathparse{10}
        \let\xd\pgfmathresult
      \pgfmathparse{11}
        \let\xqi\pgfmathresult
        \let\xq\pgfmathresult
      \pgfmathparse{-14}
        \let\bas\pgfmathresult
      \pgfmathparse{31}
        \let\haut\pgfmathresult
      \pgfmathparse{\xqi}
        \let\delta\pgfmathresult
      \draw [dashed] (0,\xa+\delta) -- (\xq,\xa);
      \draw [thick] (-0.2,\xa+\delta) -- (0.2,\xa+\delta);
      \draw (-0.5,\xa+\delta) node[left] {$X_1$};
      \draw [dashed] (0,\xb+\delta) -- (\xq,\xb);
      \draw [thick] (-0.2,\xb+\delta) -- (0.2,\xb+\delta);
      \draw (-0.5,\xb+\delta) node[left] {$X_2$};
      \draw [dashed, very thick] (0,\xc+\delta) -- (\xq,\xc);
      \draw [thick] (-0.2,\xc+\delta) -- (0.2,\xc+\delta);
      \draw (-0.5,\xc+\delta) node[left] {$X_3$};
      \draw [dashed] (0,\xd+\delta) -- (\xq,\xd);
      \draw [thick] (-0.2,\xd+\delta) -- (0.2,\xd+\delta);
      \draw (-0.5,\xd+\delta) node[left] {$X_4$};
      \draw (0,\bas) node[below] {$0$};
      \draw [-] (0,\bas) -- (0,\haut);
      \draw [dotted] (\xq,\bas) -- (\xq,\xd);
      \draw (\xq,\bas) node[below] {$q_1$};
      \pgfmathparse{7}
        \let\deltai\pgfmathresult
      \pgfmathparse{\xb-\deltai} \let\xai\pgfmathresult
      \pgfmathparse{\xc-\deltai} \let\xbi\pgfmathresult
      \pgfmathparse{\xd-\deltai} \let\xci\pgfmathresult
      \pgfmathparse{\xa+3*\deltai} \let\xdi\pgfmathresult
      \pgfmathparse{\xq+\deltai} \let\xqi\pgfmathresult
      \draw [thick] (\xq,\xa) -- (\xqi,\xdi);
      \draw [thick] (\xq,\xb) -- (\xqi,\xai);
      \draw [very thick] (\xq,\xc) -- (\xqi,\xbi);
      \draw [thick] (\xq,\xd) -- (\xqi,\xci);
      \draw [dashed] (0,\xdi+\delta+\deltai) -- (\xqi,\xdi);
      \draw [thick] (-0.2,\xdi+\delta+\deltai) -- (0.2,\xdi+\delta+\deltai);
      \draw (-0.5,\xdi+\delta+\deltai) node[left] {$X_5$};
      \pgfmathparse{\xqi+(\haut-\xdi)/3} \let\abscisse\pgfmathresult
      \draw [dashed] (\xqi,\xdi) -- (\abscisse,\haut);
      \draw [dotted] (\xqi,\bas) -- (\xqi,\xdi);
      \draw (\xqi,\bas) node[below] {$q_2$};
      \pgfmathparse{6}
        \let\deltaii\pgfmathresult
      \pgfmathparse{\xbi-\deltaii} \let\xaii\pgfmathresult
      \pgfmathparse{\xci-\deltaii} \let\xbii\pgfmathresult
      \pgfmathparse{\xdi-\deltaii} \let\xcii\pgfmathresult
      \pgfmathparse{\xai+3*\deltaii} \let\xdii\pgfmathresult
      \pgfmathparse{\xqi+\deltaii} \let\xqii\pgfmathresult
      \draw [thick] (\xqi,\xai) -- (\xqii,\xdii);
      \draw [very thick] (\xqi,\xbi) -- (\xqii,\xaii);
      \draw [thick] (\xqi,\xci) -- (\xqii,\xbii);
      \draw [thick] (\xqi,\xdi) -- (\xqii,\xcii);
      \pgfmathparse{\xqii+(\haut-\xdii)/3} \let\abscisse\pgfmathresult
      \draw [dashed] (\xqii,\xdii) -- (\abscisse,\haut);
      \pgfmathparse{\xqii-(\haut-\xdii)} \let\abscisse\pgfmathresult
      \draw [dashed] (\abscisse,\haut) -- (\xqii,\xdii);
      \draw [dotted] (\xqii,\bas) -- (\xqii,\xdii);
      \draw (\xqii,\bas) node[below] {$q_3$};
      \pgfmathparse{8.5}
        \let\deltaiii\pgfmathresult
      \pgfmathparse{\xbii-\deltaiii} \let\xaiii\pgfmathresult
      \pgfmathparse{\xcii-\deltaiii} \let\xbiii\pgfmathresult
      \pgfmathparse{\xdii-\deltaiii} \let\xciii\pgfmathresult
      \pgfmathparse{\xaii+3*\deltaiii} \let\xdiii\pgfmathresult
      \pgfmathparse{\xqii+\deltaiii} \let\xqiii\pgfmathresult
      \draw [very thick] (\xqii,\xaii) -- (\xqiii,\xdiii);
      \draw [thick] (\xqii,\xbii) -- (\xqiii,\xaiii);
      \draw [thick] (\xqii,\xcii) -- (\xqiii,\xbiii);
      \draw [thick] (\xqii,\xdii) -- (\xqiii,\xciii);
      \pgfmathparse{\xqiii+(\haut-\xdiii)/3} \let\abscisse\pgfmathresult
      \draw [very thick, dashed] (\xqiii,\xdiii) -- (\abscisse,\haut);
      \draw (\abscisse,\haut) node[below right] {$\ux^*_3$};
      \pgfmathparse{\xqiii-(\haut-\xdiii)} \let\abscisse\pgfmathresult
      \draw [dashed] (\abscisse,\haut) -- (\xqiii,\xdiii);
      \draw [dotted] (\xqiii,\bas) -- (\xqiii,\xdiii);
      \draw (\xqiii,\bas) node[below] {$q_4$};
      \pgfmathparse{7.5}
        \let\deltaiv\pgfmathresult
      \pgfmathparse{\xbiii-\deltaiv} \let\xaiv\pgfmathresult
      \pgfmathparse{\xciii-\deltaiv} \let\xbiv\pgfmathresult
      \pgfmathparse{\xdiii-\deltaiv} \let\xciv\pgfmathresult
      \pgfmathparse{\xaiii+3*\deltaiv} \let\xdiv\pgfmathresult
      \pgfmathparse{\xqiii+\deltaiv} \let\xqiv\pgfmathresult
      \draw [thick] (\xqiii,\xaiii) -- (\xqiv,\xdiv);
      \draw [thick] (\xqiii,\xbiii) -- (\xqiv,\xaiv);
      \draw [thick] (\xqiii,\xciii) -- (\xqiv,\xbiv);
      \draw [thick] (\xqiii,\xdiii) -- (\xqiv,\xciv);
      \draw [dotted] (\xqiv,\bas) -- (\xqiv,\xdiv);
      \draw (\xqiv,\bas) node[below] {$q_5$};
  \end{tikzpicture}
\caption{Example of combined graph of a quasi-regular
$(4,0)$-system over an interval $[q_1,q_5]$,
with the trajectory of an ideal point $\ux^*_3$ enlightened.}
\end{figure}
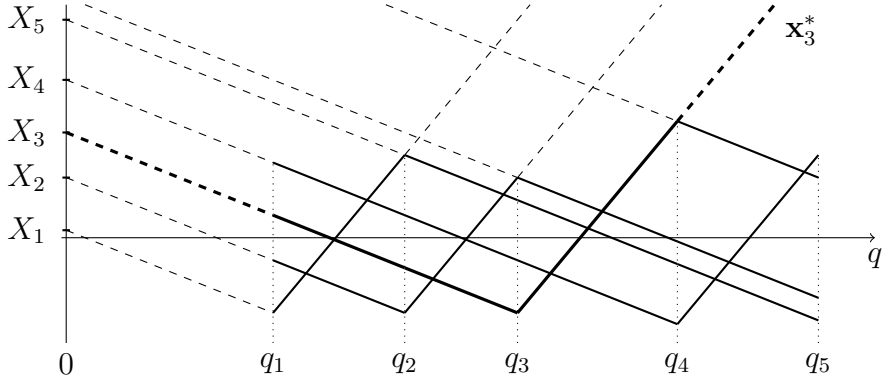

A general $(n,0)$-system also comes with a partition of its
domain into subintervals above which its combined graph consists
of a line segment of slope $n-1$ and $n-1$ line segments of
slope $-1$, but there is more flexibility in the way in which these line
segments connect the points above the left and the right
end-points of the subintervals.
In the case of a quasi-regular $(n,0)$-system, the line segments
of slope $n-1$ always connect the lowest point on the left
to the highest point on the right.

The main result of this paper is the following statement where
$\|\ \|_\infty$ stands for the maximum norm.

\begin{theorem}
 \label{intro:thm}
Let $\uP\colon[q_1,\infty)\to\Delta_n$ be a quasi-regular
$(n,0)$-system with mesh at least $\log 4$.
Then there exists a unit vector $\uu$ of $\bR^n$ such that
\[
 \|\uP(q)-\uL_\uu(q)\|_\infty \le 2n^2
 \quad
 (q\ge q_1).
\]
\end{theorem}

To say a word about the proof, recall that each convex body $\cC$
of $\bR^n$ induces a \emph{distance function} on $\bR^n$.  It is
the map from $\bR^n$ to $[0,\infty)$ which assigns to each point
$\ux$ of $\bR^n$ the smallest real number $\lambda\ge 0$, denoted
$\lambda(\ux,\cC)$, such that $\ux\in\lambda\cC$ (see \cite[\S1.3]{GL1987}).
Usually, $\cC$ is fixed and $\ux$ varies.  Here, the situation
is reversed.  The point $\ux\in\bR^n$ is fixed and we let
the convex body $\cC$ vary within the family $\cC_\uu(Q)$ with
$Q\ge 1$, for some unit vector $\uu$ of $\bR^n$.  In view of the
definition of $\cC_\uu(Q)$, we have
\begin{equation}
 \label{intro:eq:lambda(x,C)}
 \lambda(\ux,\cC_\uu(Q))
  = \max\big\{ \|\ux\|Q^{-1},\, |\ux\cdot\uu| Q^{n-1} \big\}
  \quad
  (Q\ge 1).
\end{equation}
Suppose that the coordinates of $\uu$ are linearly independent
over $\bQ$ and that $\ux\in\bZ^n\setminus\{0\}$.  Then, we have
$0<|\ux\cdot\uu|<\|\ux\|$ and we define a map
$L_\ux\colon[0,\infty)\to\bR$ by
\begin{align*}
 L_\ux(q)
   :=L(\ux,q)
   := &\log\lambda(\ux,\cC_\uu(e^q)) \\
    = &\max\big\{ \log\|\ux\|-q,\, \log|\ux\cdot\uu|+(n-1)q \big\}
 \quad
 (q\ge 0).
\end{align*}
Its graph is a polygon with two sides: a line segment of slope
$-1$ followed by an half-line with slope $n-1$.  The function
$L_\ux$ is continuous and has a local minimum at the point where its
graph changes slope from $-1$ to $n-1$. Although $\ux$ is fixed,
we say that $L_\ux$, or its graph, represents the \emph{trajectory}
of the point $\ux$ with respect to the varying family of convex
bodies $\cC_\uu(Q)$.  Clearly, this trajectory is uniquely determined
by its local minimum.  It is not difficult to show that the combined graph of
$\uL_\uu$ above any compact interval is covered by the trajectories
of finitely many non-zero integer points (see \cite[\S4]{SS2009}).

Now, let $\uP\colon[q_1,\infty)\to\Delta_n$ be a quasi-regular
$(n,0)$-system.  In the notation
of Definition \ref{def}, we can imagine its combined graph covered
by the trajectories of
a sequence of ``ideal points'' $\ux^*_i$ having local minima
at $(q_i,P_1(q_i))$.  Figure 1 shows the trajectory
of such an ideal point $\ux^*_3$.  In general, we cannot hope
for such points to exist.
Instead, we construct a sequence $(\ux_i)_{i\ge 1}$ of
integer points and a unit vector $\uu$ such that, for each
$i\ge1$, the trajectory of $\ux_i$ is close
to ideal and moreover the $n$-tuple
$(\ux_i,\dots,\ux_{i+n-1})$ is a basis of $\bZ^n$.
In practice, the vector $\uu$ is also constructed as a
limit of unit vectors $\uu_i$ where $\uu_i$ is perpendicular
to $\ux_{i},\dots,\ux_{i+n-2}$ for each $i\ge 1$.  Then, it
suffices to choose the sequence $(\ux_i)_{i\ge 1}$ so that the trajectory
of $\ux_i$ with respect to the family $\cC_{\uu_{i+1}}(Q)$ is close to ideal.
To this end, we require $\uP$ to have mesh at least $\log 4$.
This allows us to control appropriately the norms of the points
$\ux_i$ as well as the angles that they make with respect to
certain subspaces.

%
%

\section{Almost orthogonal sequences}

For each $k=1,\dots,n$, we endow $\bigwedge^k\bR^n$ with
the Euclidean space structure characterized by the property that,
for any orthonormal basis $(\ue_1,\dots,\ue_n)$ of $\bR^n$, the
products $\ue_{j_1}\wedge\cdots\wedge\ue_{j_k}$ with $1\le
j_1<\cdots<j_k\le n$ form an orthonormal basis of
$\bigwedge^k\bR^n$.  We denote by $\|\up\|$ the
associated norm of an element $\up$ of $\bigwedge^k\bR^n$.
We also denote by $\bigwedge^k\bZ^n$ the lattice of
$\bigwedge^k\bR^n$ of co-volume $1$ spanned by the
products $\ux_1\wedge\cdots\wedge\ux_k$ with $\ux_1,\dots,\ux_k
\in \bZ^n$.

The projective distance between two non-zero points
$\ux,\uy$ of $\bR^n$ is
\[
 \dist(\ux,\uy):=\frac{\|\ux\wedge\uy\|}{\|\ux\|\,\|\uy\|}.
\]
It depends only on the classes of $\ux$ and $\uy$ in
$\bP^{n-1}(\bR)$ and represents the sine of the angle between
the one-dimensional subspaces of $\bR^n$ spanned by $\ux$ and $\uy$.
This function induces a metric on $\bP^{n-1}(\bR)$ (satisfying
the triangle inequality) and $\bP^{n-1}(\bR)$ is complete with
respect to that metric.

Given a point $\ux$ of $\bR^n$ and a subspace $U$
of $\bR^n$, we denote by $U^\perp$ the orthogonal complement
of $U$ and by $\proj_U(\ux)$ the orthogonal projection of
$\ux$ on $U$.  If $\ux$ is non-zero, we also define
\[
 \dist(\ux,U):=\frac{\|\proj_{U^\perp}(\ux)\|}{\|\ux\|}.
\]
The next lemma connects the two notions of distance.

\begin{lemma}
\label{lem0}
If $\ux$ is a non-zero point of $\bR^n$, and if $U$ is a non-zero
proper subspace of $\bR^n$ with basis $(\uy_1,\dots,\uy_k)$, then
\[
 \dist(\ux,U)
  =\frac{\|\ux\wedge\uy_1\wedge\cdots\wedge\uy_k\|}
        {\|\ux\|\,\|\uy_1\wedge\cdots\wedge\uy_k\|}
  = \min\{\dist(\ux,\uy)\,;\,\uy\in U\setminus\{0\}\,\}.
\]
\end{lemma}

\begin{proof} The first formula follows from the definition using
\[
 \|\ux\wedge\uy_1\wedge\cdots\wedge\uy_k\|
 =\|\proj_{U^\perp}(\ux)\wedge\uy_1\wedge\cdots\wedge\uy_k\|
 =\|\proj_{U^\perp}(\ux)\|\,\|\uy_1\wedge\cdots\wedge\uy_k\|.
\]
It implies in particular that $\dist(\ux,\uy) =
\dist(\ux,\langle\uy\rangle_\bR)$ for any $\uy \in
\bR^n\setminus\{0\}$.
To prove the second equality of the lemma, we first note
that, for any subspace $V$ of $U$, we have $\proj_{U^\perp}(\ux)=
\proj_{U^\perp}(\proj_{V^\perp}(\ux))$ and so
$\dist(\ux,U) \le \dist(\ux,V)$.  In particular, this implies
that $\dist(\ux,U)\le \dist(\ux,\uy)$ for any $\uy\in
U\setminus\{0\}$.  If $\ux\notin U^\perp$, then
$\uy:=\proj_U(\ux)$ is a non-zero element of $U$
with $\dist(\ux,U)=\dist(\ux,\uy)$ because $\ux$ has the
same orthogonal projection on $U^\perp$ as on
$\langle\uy\rangle_\bR^\perp$.  Thus the second equality
holds in that case. If $\ux\in U^\perp$, then it still holds
because $\dist(\ux,U)=1=\dist(\ux,\uy)$
for any $\uy\in U\setminus\{0\}$.
\end{proof}

\begin{definition}
We say that a sequence $(\ux_1,\dots,\ux_k)$ of vectors of
$\bR^n$ is \emph{almost orthogonal} if it is linearly independent
and if
\[
 \dist(\ux_j,\langle\ux_1,\dots,\ux_{j-1}\rangle_\bR)\ge 1/2
 \quad
 (2\le j \le k).
\]
\end{definition}

By Lemma \ref{lem0}, it follows that any subsequence of
an almost orthogonal sequence is almost orthogonal.  Moreover,
if $(\ux_1,\dots,\ux_k)$ is almost orthogonal, then
\[
 \|\ux_1\wedge\cdots\wedge\ux_k\|
 = \|\ux_1\|
   \prod_{j=2}^k
   \left( \|\ux_j\|\,
         \dist\big(\ux_j,\langle\ux_1,\dots,\ux_{j-1}\rangle_\bR\big)
   \right)
 \ge 2^{-(k-1)} \|\ux_1\|\cdots\|\ux_k\|.
\]
Note that in \cite{R_preprint}, we use a stronger
notion of almost orthogonality.

We say that an element $\ux$ of $\bZ^n$ is \emph{primitive} if
it is non-zero and if its coordinates are relatively prime as
a set.  More generally, we say that a $k$-tuple $(\ux_1,\dots,\ux_k)$
of elements of $\bZ^n$ is \emph{primitive} if $\ux_1 \wedge \cdots
\wedge\ux_k$ is non-zero and if its coordinates with respect to
a basis of $\bigwedge^k\bZ^n$ are relatively prime.  This condition
is equivalent to asking that $(\ux_1,\dots,\ux_k)$ can be extended to a
basis $(\ux_1,\dots,\ux_n)$ of $\bZ^n$.  In particular, it requires
that $1\le k\le n$.

Finally, we say that a non-zero subspace $U$ of $\bR^n$ is
\emph{defined over $\bQ$} if it is spanned by elements
of $\bQ^n$.  Following Schmidt in \cite{Sc1967}, we then define
the \emph{height} of $U$ by
\[
 H(U) = \|\ux_1\wedge\cdots\wedge\ux_k\|
\]
where $(\ux_1,\dots,\ux_k)$ is any basis of $U\cap\bZ^n$.
This is independent of the choice of the basis.  The next result
summarizes some of the above considerations.

\begin{lemma}
 \label{lem1}
Let $(\ux_1,\dots,\ux_{n-1})$ be an almost orthogonal primitive
$(n-1)$-tuple of points of $\bZ^n$ and let
$U:=\langle\ux_1,\dots,\ux_{n-1}\rangle_\bR$.  Then, we have
\[
 2^{-(n-2)}\|\ux_1\|\cdots\|\ux_{n-1}\|
 \le H(U)
 \le \|\ux_1\|\cdots\|\ux_{n-1}\|.
\]
\end{lemma}

We conclude this section with a particular construction of
almost orthogonal sequences.  It will serve as the initial step for
a recursive construction of integer points in the next section.

\begin{lemma}
\label{lem3}
Let $(\ue_1,\dots,\ue_n)$ denote the canonical basis of $\bZ^n$
and let $B_1,\dots,B_{n-1}\in\bZ$ with $B_i\ge 2^{i-1}$
for $i=1,\dots,n-1$.  Set
\[
 \ux_i=B_i\ue_i+\ue_{i+1}
 \quad
 (i=1,\dots,n-1).
\]
Then $(\ux_1,\dots,\ux_{n-1})$ is an almost orthogonal primitive
$(n-1)$-tuple of integer points.
\end{lemma}

\begin{proof}
We first note that $(\ue_1,\ux_1,\dots,\ux_{n-1})$ is a basis of
$\bZ^n$ and therefore $(\ux_1,\dots,\ux_{n-1})$ is primitive. Let $k$
be an integer with $2\le k\le n-1$.  Since
\[
 \ux_1\wedge\cdots\wedge\ux_k\wedge\ue_{k+1}
  = B_1\cdots B_k\ue_1\wedge\cdots\wedge\ue_{k+1},
\]
we must have $\|\ux_1\wedge\cdots\wedge\ux_k\| \ge B_1\cdots B_k$.
As we also have
\[
 \|\ux_1\wedge\cdots\wedge\ux_{k-1}\|\, \|\ux_k\|
  \le \prod_{i=1}^k \|\ux_i\|
   = \prod_{i=1}^k \sqrt{1+B_i^2}
  \le \prod_{i=1}^k \Big(B_i\exp\Big(\frac{1}{2B_i^2}\Big)\Big)
  \le 2\prod_{i=1}^k B_i\,,
\]
we conclude from Lemma \ref{lem0} that
$\dist(\ux_k,\langle\ux_1,\dots,\ux_{k-1}\rangle_\bR)
  \ge 1/2$.
This shows that the sequence $(\ux_1,\dots,\ux_{n-1})$ is almost
orthogonal.
\end{proof}

%
%

\section{A recursive construction of points}

The next lemma is the key to a recursive construction of
points in $\bZ^n$ which is at the heart of the proof of our
main theorem.

\begin{lemma}
 \label{lem2}
Let $(\uy_1,\dots,\uy_{n-1})$ be an almost orthogonal primitive
$(n-1)$-tuple of points of $\bZ^n$ and let $A$ be a real number
with $A \ge 2+\|\uy_1\|+\cdots+\|\uy_{n-1}\|$.
Then, there exists a point $\uy_n\in\bZ^n$ with the following
properties
\begin{itemize}
 \item[1)] $A \le \|\uy_n\| \le 2A$,
 \item[2)] $(\uy_1,\uy_2,\dots,\uy_n)$ is a basis of $\bZ^n$,
 \item[3)] $(\uy_2,\dots,\uy_n)$ is almost orthogonal,
 \item[4)] if $\uu$ is a unit vector perpendicular to
   $U:=\langle\uy_1,\dots,\uy_{n-1}\rangle_\bR$,
   and if $\uu'$ is a unit vector perpendicular
   to $U':=\langle\uy_2,\dots,\uy_n\rangle_\bR$, then
   \[
    \dist(\uu,\uu') \le \frac{1}{A\,H(U)}
    \et
    |\uy_1\cdot\uu'| = \frac{1}{H(U')}.
   \]
\end{itemize}
\end{lemma}

\begin{proof}
Let $U$ and $\uu$ be as in the condition 4). We define
$V=\langle\uy_2,\dots,\uy_{n-1}\rangle_\bR$, and choose
a unit vector $\uv$ of $U$ which is perpendicular to $V$.
Then $(\uu,\uv)$ is an orthonormal basis for $V^\perp$.

The hyperplane $H(U)^{-1}\uu+U$ is a closest translate of $U$
which contains a point of $\bZ^n$ not in $U$.  For any point
$\uy$ of this hyperplane, we have $|\det(\uy_1,\dots,\uy_{n-1},\uy)|=1$
and there exist $\epsilon_1,\dots,\epsilon_{n-1} \in [-1/2,1/2]$
such that
\[
 \uy+\epsilon_1\uy_1+\cdots+\epsilon_{n-1}\uy_{n-1}
 \in \bZ^n.
\]
We apply this to the point $\uy=H(U)^{-1}\uu+(3/2)A\uv$.  This yields
an integer point
\[
 \uy_n:=\frac{1}{H(U)}\uu+\frac{3}{2}A\uv
        +\epsilon_1\uy_1+\cdots+\epsilon_{n-1}\uy_{n-1}
      \in \bZ^n
\]
for which $(\uy_1,\dots,\uy_n)$ is a basis of $\bZ^n$ because
$|\det(\uy_1,\dots,\uy_n)|=1$.  Since $H(U)\ge 1$, we also find
\[
 \Big\|\uy_n-\frac{3}{2}A\uv\Big\|
  \le 1+\frac{1}{2}\big(\|\uy_1\|+\cdots+\|\uy_{n-1}\|\big)
  \le \frac{A}{2}
\]
and thus $A\le \|\uy_n\|\le 2A$.  This shows that the
conditions 1) and 2) hold.

Since the orthogonal projection of $\uy_n$ on $V^\perp$
has norm at least
\begin{equation}
 \label{lem2:eq}
 |\uy_n\cdot\uv|
 = \Big|\frac{3}{2}A+\epsilon_1\uy_1\cdot\uv\Big|
 \ge \frac{3}{2}A - \frac{1}{2}\|\uy_1\|
 \ge A,
\end{equation}
we find that
\[
 \dist(\uy_n,\langle\uy_2,\dots,\uy_{n-1}\rangle_\bR)
 = \dist(\uy_n,V)
 = \frac{\|\proj_{V^\perp}(\uy_n)\|}{\|\uy_n\|}
 \ge \frac{A}{\|\uy_n\|}
 \ge \frac{1}{2}.
\]
We also note that
\[
 \dist(\uy_i,\langle\uy_2,\dots,\uy_{i-1}\rangle_\bR)
 \ge \dist(\uy_i,\langle\uy_1,\dots,\uy_{i-1}\rangle_\bR)
 \ge \frac{1}{2}
 \quad
 (3\le i\le n-1)
\]
because $(\uy_1,\dots,\uy_{n-1})$ is almost orthogonal.  Thus
$(\uy_2,\dots,\uy_n)$ is almost orthogonal as well, and so
the condition 3) holds.

Let $U':=\langle\uy_2,\dots,\uy_n\rangle_\bR$ and let $\uu'$
be a unit vector perpendicular to $U'$.  Since $V\subset U'$,
we have $\uu'\in V^\perp$ and so we can write
\[
 \uu'=a\uu+b\uv
\]
for some $a,b\in\bR$ with $a^2+b^2=1$.  Since $\uy_n\in U'$,
we have $0 =\uy_n\cdot\uu'$ and so
\[
 |b| = |a| \frac{|\uy_n\cdot\uu|}{|\uy_n\cdot\uv|}
     \le \frac{|\uy_n\cdot\uu|}{A}
     = \frac{1}{A H(U)}
\]
where the middle inequality uses \eqref{lem2:eq} and $|a|\le 1$.
We conclude that
\[
 \dist(\uu,\uu')
  = \|\uu\wedge\uu'\|
  = \| b\uu\wedge\uv \|
  = |b|
  \le \frac{1}{AH(U)}.
\]
Finally, we find that
\[
 1=|\det(\uy_1,\dots,\uy_n)|
  =|\uy_1\cdot\uu'|\,\|\uy_2\wedge\cdots\wedge\uy_n\|
  =|\uy_1\cdot\uu'| H(U')
\]
and so $|\uy_1\cdot\uu'|=H(U')^{-1}$.
\end{proof}

\begin{proposition}
\label{prop}
Let $(A_i)_{i\ge 1}$ be a sequence of real numbers with $A_1\ge 1$
and $A_{i+1}\ge 4A_i$ for each $i\ge 1$.  Then there exist a sequence
of points $(\ux_i)_{i\ge 1}$ in $\bZ^n$ and a unit vector $\uu$ of
$\bR^n$ which, for each index $i\ge 1$, fulfil the following conditions:
\begin{itemize}
 \item[1)] $(\ux_i,\ux_{i+1},\dots,\ux_{i+n-1})$ is a basis of $\bZ^n$,
 \item[2)] $A_i\le \|\ux_i\|\le 2A_i$,
 \item[3)] $2^{-n} \le |\ux_i\cdot\uu|\,A_{i+1}\cdots A_{i+n-1} \le 2^n$.
\end{itemize}
\end{proposition}

\begin{proof}
We first construct an almost orthogonal primitive
$(n-1)$-tuple $(\ux_1,\dots,\ux_{n-1})$ as in Lemma \ref{lem3}
using $B_1=\lceil A_1\rceil, \dots, B_{n-1}=\lceil A_{n-1}\rceil$.
Then these points satisfy $A_i\le \|\ux_i\|\le 2A_i$ for
$i=1,\dots,n-1$. We set
\[
 U_1=\langle\ux_1,\dots,\ux_{n-1}\rangle_\bR
\]
and denote by $\uu_1$ a unit vector of $\bR^n$ orthogonal to $U_1$.
Then, using the fact that
\[
  2+2A_i+\cdots+2A_{i+n-2}
  \le
  2(1+A_1+\cdots+A_{i+n-2})
  \le
  A_{i+n-1}
  \quad
  (i\ge 1),
\]
Lemma \ref{lem2} allows us to construct recursively, for each $i\ge 1$,
an additional integer point $\ux_{i+n-1}$, an additional
($n-1$)-dimensional vector subspace $U_{i+1}$ and an additional
unit vector $\uu_{i+1}$ with the following properties
\begin{itemize}
 \item[1)] $A_{i+n-1} \le \|\ux_{i+n-1}\| \le 2A_{i+n-1}$,
 \item[2)] $(\ux_i,\dots,\ux_{i+n-1})$ is a basis of $\bZ^n$,
 \item[3)] $(\ux_{i+1},\dots,\ux_{i+n-1})$ is almost orthogonal,
 \item[4)] $U_{i+1}=\langle\ux_{i+1},\dots,\ux_{i+n-1}\rangle_\bR$
   and $\uu_{i+1}$ is perpendicular to $U_{i+1}$,
 \item[5)] $\dist(\uu_i,\uu_{i+1}) \le A_{i+n-1}^{-1}H(U_i)^{-1}$,
 \item[6)] $|\ux_i\cdot\uu_{i+1}| = H(U_{i+1})^{-1}$.
\end{itemize}
Thanks to Lemma \ref{lem1}, we have
\[
 2^{-(n-2)}\|\ux_i\|\cdots\|\ux_{i+n-2}\|
  \le H(U_i)
  \le \|\ux_i\|\cdots\|\ux_{i+n-2}\|
 \quad
 (i\ge 1),
\]
and therefore
\begin{equation}
 \label{prop:eq}
 2^{-(n-2)}A_i\cdots A_{i+n-2}
  \le H(U_i)
  \le 2^{n-1}A_i\cdots A_{i+n-2}
 \quad
 (i\ge 1).
\end{equation}
In view of the growth of the sequence $(A_i)_{i\ge 1}$, this
implies that $H(U_{i+1})\ge 2H(U_i)$ for each $i\ge 1$.
Then, using 5), we deduce that the image of $(\uu_i)_{i\ge 1}$
in $\bP^{n-1}(\bR)$ converges to the class of a unit vector
$\uu$ with
\[
 \dist(\uu_i,\uu)
 \le \sum_{j=i}^\infty \dist(\uu_j,\uu_{j+1})
 \le \sum_{j=i}^\infty \frac{1}{A_{j+n-1}H(U_j)}
 \le \frac{2}{A_{i+n-1}H(U_i)}
 \quad
 (i\ge 1).
\]
Fix an index $i\ge 1$. Upon replacing $\uu_{i+1}$ by $-\uu_{i+1}$
if necessary, we may assume that $\uu_{i+1}\cdot\uu\ge 0$.
Then, the above estimate yields
\begin{align*}
 |\ux_i\cdot\uu-\ux_i\cdot\uu_{i+1}|
 &\le \|\ux_i\|\,\|\uu-\uu_{i+1}\| \\
 &\le 2\|\ux_i\|\dist(\uu,\uu_{i+1})
 \le \frac{4\|\ux_i\|}{A_{i+n}H(U_{i+1})}
 \le \frac{1}{2H(U_{i+1})}
\end{align*}
since $A_{i+n}\ge 4^nA_i\ge 8\|\ux_i\|$.
In view of 6), this implies that
\[
 \frac{1}{2H(U_{i+1})}
   \le |\ux_i\cdot\uu|
   \le \frac{2}{H(U_{i+1})}.
\]
Using the estimates for $H(U_{i+1})$ given by
\eqref{prop:eq}, this shows that the third condition
of the proposition is satisfied.
\end{proof}

In view of the formula \eqref{intro:eq:lambda(x,C)} for
$\lambda(\ux,\cC_\uu(Q))$, the estimates of the
proposition yield the following result.

\begin{corollary}
 \label{cor}
Let the notation be as in the proposition.  For each
integer $i\ge 1$ and each real number $Q\ge 1$, we have
\[
 2^{-n}\frac{A_i}{Q}\max\left\{1,\, \frac{Q}{Q_i}\right\}^n
   \le \lambda(\ux_i,\cC_\uu(Q))
   \le 2^n\frac{A_i}{Q}\max\left\{1,\, \frac{Q}{Q_i}\right\}^n.
\]
where $Q_i=(A_i\cdots A_{i+n-1})^{1/n}$.
\end{corollary}

\section{Proof of the main theorem}
\label{sec:proof}

To deduce our main theorem from Proposition \ref{prop}
and its corollary, we simply use the following well-known
principle.

\begin{lemma}
 \label{lem4}
Let $\cC$ be a convex body of $\bR^n$ and let $\uy_1,\dots,\uy_n$
be linearly independent points of $\bZ^n$.  Suppose that
\begin{equation}
 \label{lem4:eq}
 \lambda(\uy_1,\cC)\cdots\lambda(\uy_n,\cC)\vol(\cC) \le B
\end{equation}
for some real number $B$.  Then, we have
\[
 \big(\lambda_1(\cC),\dots,\lambda_n(\cC)\big)
 \le \Phi_n\big(\lambda(\uy_1,\cC),\dots,\lambda(\uy_n,\cC)\big)
 \le \frac{n!B}{2^n}\big(\lambda_1(\cC),\dots,\lambda_n(\cC)\big),
\]
where the inequality is meant component-wise.
\end{lemma}

\begin{proof}
Choose a permutation $\sigma\in S_n$ such that
$\lambda(\uy_{\sigma(1)},\cC) \le \cdots \le
\lambda(\uy_{\sigma(n)},\cC)$.  By definition of the successive
minima, we have $\lambda_j(\cC) \le \lambda(\uy_{\sigma(j)},\cC)$
for $j=1,\dots,n$. As Minkowski's second convex body theorem gives
\[
 \frac{2^n}{n!} \le \lambda_1(\cC)\cdots\lambda_n(\cC)\vol(\cC),
\]
comparison with \eqref{lem4:eq} yields
\[
 \lambda_j(\cC)
  \le \lambda(\uy_{\sigma(j)},\cC)
  \le \frac{n!B}{2^n}\lambda_j(\cC)
 \quad
 (1\le j\le n).
\qedhere
\]
\end{proof}

\medskip
\begin{proof}[\textbf{Proof of Theorem \ref{intro:thm}}]
Let $(X_i)_{i\ge1}$ and $(q_i)_{i\ge 1}$ be as in
Definition \ref{def}, for the given quasi-regular
$(n,0)$-system $\uP$.  We define
\[
 A_i:=\exp(X_i)
 \quad
 (i\ge 1).
\]
For this choice of parameters, we select a sequence of integer
points $(\ux_i)_{i\ge 1}$ and a unit vector $\uu$ which satisfy
the conclusion of Proposition \ref{prop}.  We also define
\[
 L(\ux_i,q) := \log\lambda(\ux_i,\cC_\uu(e^q))
 \quad
 (q\ge 0,\ i\ge 1).
\]
Since $\exp(q_j)=(A_j\cdots A_{j+n-1})^{1/n}$ for each $j\ge 1$,
Corollary \ref{cor} yields
\begin{equation}
 \label{proof:eq:L}
 | L(\ux_j,q) -X_j-n\max\{0,q-q_j\}+q | \le n\log(2)
 \quad
 (q\ge 0,\ j\ge 1).
\end{equation}
To show that the vector $\uu$ has the required property, we fix
an integer $i\ge 1$ and a real number $q\in[q_i,q_{i+1}]$.
The points $\ux_i,\dots,\ux_{i+n-1}$ form a basis of $\bZ^n$ and,
since $q_i\le q\le q_{i+1}$, the estimates \eqref{proof:eq:L}
show that they satisfy
\begin{align*}
 &|L(\ux_i,q)-X_i-n(q-q_i)+q|\le n\log2, \\
 &|L(\ux_{i+1},q)-X_{i+1}+q|\le n\log2, \\
 &\quad \cdots\\
 &|L(\ux_{i+n-1},q)-X_{i+n-1}+q|\le n\log2.
\end{align*}
On one hand, these inequalities give
\[
 \|\uP(q)-\Phi_n(L(\ux_i,q),\dots,L(\ux_{i+n-1},q))\|_\infty
 \le n\log 2.
\]
On the other hand, since $\vol(\cC(e^q))\le 2^n$, they
also lead to
\[
 L(\ux_i,q)+\cdots+L(\ux_{i+n-1},q)+\log\vol(\cC(e^q))
 \le
 (n^2+n)\log2
\]
which, by Lemma \ref{lem4}, implies that
\[
 \|\uL_\uu(q)-\Phi_n(L(\ux_i,q),\dots,L(\ux_{i+n-1},q))\|_\infty
 \le
 (n^2+n)\log2 + \log(n!/2^n).
\]
This gives $\|\uP(q)-\uL_\uu(q)\|_\infty \le (n^2+n)\log(2)+\log(n!)
\le 2n^2$, as requested.
\end{proof}

\subsection*{Acknowledgements}
Part of this research was done while the author was visiting Professor
Wolfgang Schmidt at the University of Colorado in Boulder for a week
in February 2013.  He warmly thanks him for his invitation and for suggesting
the problem discussed here.


\end{document}